\documentclass[12pt]{article}

\usepackage{latexsym,amsmath,amscd,amssymb,graphics,float}
\usepackage{enumerate}
\usepackage{multirow}
\usepackage{graphicx,lscape}

\usepackage[colorlinks]{hyperref}
\usepackage{url}

\begin{document}

\newenvironment{proof}[1][Proof]{\textbf{#1.} }{\ \rule{0.5em}{0.5em}}

\newtheorem{theorem}{Theorem}[section]
\newtheorem{definition}[theorem]{Definition}
\newtheorem{lemma}[theorem]{Lemma}
\newtheorem{remark}[theorem]{Remark}
\newtheorem{proposition}[theorem]{Proposition}
\newtheorem{corollary}[theorem]{Corollary}
\newtheorem{example}[theorem]{Example}

\numberwithin{equation}{section}
\newcommand{\ep}{\varepsilon}
\newcommand{\R}{{\mathbb  R}}
\newcommand\C{{\mathbb  C}}
\newcommand\Q{{\mathbb Q}}
\newcommand\Z{{\mathbb Z}}
\newcommand{\N}{{\mathbb N}}

\newcommand{\bfi}{\bfseries\itshape}

\newsavebox{\savepar}
\newenvironment{boxit}{\begin{lrbox}{\savepar}
\begin{minipage}[b]{15.5cm}}{\end{minipage}\end{lrbox}
\fbox{\usebox{\savepar}}}

\title{{\bf On the rattleback dynamics}}
\author{R\u{a}zvan M. Tudoran and Anania G\^\i rban}

\date{}
\maketitle \makeatother

\begin{abstract}
In this paper we present some relevant dynamical properties of an idealized conservative model of the rattleback, from the Poisson dynamics point of view. In the first half of the article, along with a dynamical study of the orbits, using a Hamilton--Poisson realization of the dynamical system, we provide a geometric characterization of the space of orbits in terms of Whitney stratifications associated to the image of the energy--Casimir mapping. In the second half of the article we provide an explicit method to stabilize asymptotically any arbitrary fixed orbit/cycle of the rattleback system and to keep unchanged the geometry of the model space.
\end{abstract}

\medskip

\textbf{AMS 2010}: 70H05; 37J25; 37J35.


\section{Introduction}
\label{section:one}

The \textit{rattleback}, also known as the \textit{celtic stone} (or simply, the \textit{celt}), is a rigid semi-ellipsoidal (canoe shape) object which exhibits rotational preference. More precisely, when spun in one direction, after a few seconds it starts pitching (``rattles"), stops spinning, and then slowly begins spinning in the opposite direction. Surprisingly, when spun in the other direction, it seems to keep spinning steadily. The existence of objects exhibiting this unintuitive property, dates back at least to the Iron Age, as being mentioned by archeologists studying Celtic sites. Starting from 1895 with the article \cite{walker}, these objects are the subject of many research papers trying to explain scientifically their unintuitive dynamical properties (see, e.g. \cite{rattle2}, \cite{bkk}, \cite{rattle3}, \cite{rattle4}, \cite{tokieda1}, \cite{kondo}, \cite{rauch}, \cite{bm}, and the references therein). 

In this paper we continue the analysis of an idealized conservative model of the rattleback, introduced in \cite{rattle2} and then studied further in \cite{tokieda1}. As this model is governed by a 3-dimensional system of first order differential equations, from now on, we will refer to this mathematical model of the rattlebak, as the \textit{rattlebak system}. Our contribution to the study of the rattleback system consists of a complete dynamical analysis from the Hamiltonian/conservative point of view. More precisely, the structure of this article is the following. In the second section we provide a family of Hamilton-Poisson realizations of the rattleback system, parameterized by $SL(2,\mathbb{R})$. In oder to have a self contained presentation, in the next two sections we discuss two subjects already analyzed in \cite{tokieda1} (using different approaches), concerning the equilibrium states of the system and their Lyapunov stability (discussed in Section 3), and the existence of periodic orbits (discussed in Section 4). In the fifth section we present some geometric properties of the image of the energy-Casimir mapping (associated to a specific Hamilton-Poisson realization of the rattleback system), and define some naturally associated semialgebraic splittings of the image. More precisely, we discuss the relation between the image through the energy-Casimir mapping of the families of equilibrium states of the rattleback system, and the canonical Whitney stratifications of the semialgebraic splittings of the image of the energy-Casimir mapping. In the sixth section we give a complete topological classification of the fibers of the energy-Casimir mapping, classification that follows naturally from the stratifications introduced in the previous section. The seventh section is dedicated to the analysis of heteroclinic solutions of the rattleback system. In the eighth section we give a Lax formulation of the rattleback system. In the last section of this article, we provide a method to stabilize asymptotically any arbitrary fixed orbit/cycle of the rattleback system. More precisely, for an a-priori fixed orbit/cycle, $\mathcal{O}$, we construct explicitly a conservative perturbation (in the sense that the Casimir invariant, $C$, of the Poisson model space, remains a first integral for the perturbed system) which preserves the orbit/cycle $\mathcal{O}$, keeps its dynamical nature unchanged, and moreover, the orbit/cycle $\mathcal{O}$ becomes asymptotically stable (with respect to perturbations along the invariant manifold $C^{-1}(\{C(\mathcal{O})\})$) as an orbit of the perturbed dynamics. The explicit construction of the perturbations agrees with the topological classification of the orbits given in the sixth section. More precisely, for each arbitrary fixed dynamical object (i.e. stable equilibrium point, periodic orbit, heteroclinic cycle) we construct an explicit conservative perturbation of the rattleback system with the above mentioned properties. Moreover, in the case of the periodic orbits, we construct conservative perturbations leading to a stronger type of asymptotic stability, i.e. the so called asymptotic stability with phase.
\medskip

\section{Hamilton-Poisson realizations of the rattleback system}
The rattleback system analyzed in this work, was introduced in \cite{rattle2}, and is governed by the following system of differential equations
\begin{equation}\label{sys}
\left\{ \begin{array}{l}
 \dot x = \lambda xz \\
 \dot y = - yz \\
 \dot z = y^2 -\lambda x^2, \\
 \end{array} \right.
\end{equation}
where $\lambda$ is a positive real parameter. Let us recall from \cite{rattle2}, \cite{tokieda1}, the physical meaning of $(x,y,z)$ and the parameter $\lambda$. More precisely, $(x,y,z)$ represent the (pitching, rolling, spinning) modes of the motion, whereas $\lambda$ is a parameter related to the aspect ratio of the rattleback.

A first look at the system tells us that the planes $x=0$ and $y=0$ are both dynamically invariant sets of the system \eqref{sys}. As we shall see in the following, the rattleback system can be realized as a Hamilton-Poisson system outside the invariant plane $y=0$. In order to do that, let us recall from \cite{tokieda1} that the system \eqref{sys} is completely integrable. More precisely, the rattleback system admits two first integrals given by $I_1:=\dfrac{1}{2}\left(x^2 +y^2 +z^2\right)$, and $I_2:=xy^\lambda$. Due to the dimensionality of the system, there are two natural Hamilton-Poisson realizations, namely, we choose $I_1$ as the Hamiltonian of the system (and then $I_2$ generates the Poisson structure of the model space), or, we choose $I_2$ as Hamiltonian (and then $I_1$ generates the Poisson structure of the model space). The first Hamilton-Poisson realization (with the  Hamiltonian given by $I_1$) was used for the analysis of the rattleback system presented in \cite{tokieda1}. Throughout this work we focus on the second Hamilton-Poisson realization  (with the Hamiltonian given by $I_2)$. 

Before stating the first result of this section, for the sake of completeness, we briefly recall some generalities regarding 3D  Hamilton-Poisson systems. In order to do that, let us fix $\Omega\subseteq\mathbb{R}^{3}$ an open subset, and let $C\in {\mathcal{C}}^\infty(\Omega,\mathbb{R})$ be a smooth function. These data provide a Poisson structure on $\Omega$, generated by the Poisson bracket $\{f,g\}_{C}:=\langle\nabla C,\nabla f\times\nabla g \rangle, ~\forall f,g\in C^\infty(\Omega,\mathbb{R})$, where $\langle\cdot,\cdot\rangle$ stands for the canonical inner product on $\mathbb{R}^3$. More precisely, the matrix formulation of the associated Poisson structure, $\Pi_{C}$, is given by $\Pi_{C}:=[\{\pi_i ,\pi_j\}_{C}]_{1\leq i,j\leq 3}$, where $\pi_1, \pi_2, \pi_3 \in C^\infty(\Omega,\mathbb{R})$ stand for the coordinate functions on $\Omega$, i.e. for $i\in\{1,2,3\}$, $\pi_{i}:\Omega\rightarrow \mathbb{R}$, $\pi_i (x_1,x_2,x_3):=x_i, ~\forall(x_1,x_2,x_3)\in\Omega$. Moreover, any arbitrary fixed nonvanishing smooth function, $\nu\in {\mathcal{C}}^\infty(\Omega,\mathbb{R})$, defines a new Poisson structure on $\Omega$, given by $\nu\Pi_C$. Recall that on the Poisson manifold $(\Omega,\nu \Pi_{C})$, the Hamiltonian vector field generated by a smooth function $H\in C^\infty(\Omega,\mathbb{R})$, is given by $X_H :=\nu\Pi_C \nabla H$, where $\nabla H$ stands for the gradient vector field with respect to the canonical inner product on $\mathbb{R}^3$. Moreover, as the Lie derivative along $X_H$ is given by $\mathcal{L}_{X_{H}}(f)=\nu\{f,H\}_{C}, ~\forall f\in C^\infty(\Omega,\mathbb{R})$, it follows that $H$ is a first integral of $X_H$. Regarding first integrals, note that $C$ is a Casimir invariant of the Poisson manifold $(\Omega,\nu\Pi_{C})$, i.e. $\{C,f\}_{C}=0, ~ \forall f\in C^\infty(\Omega,\mathbb{R})$. Thus, $C$ is a first integral for every Hamiltonian vector field defined on the Poisson manifold $(\Omega,\nu\Pi_{C})$. Consequently, any Hamiltonian vector field $X_{H}$ defined on the Poisson manifold $(\Omega,\nu \Pi_{C})$, admits $H$ and $C$ as first integrals. Conversely, any 3D system of first order differential equations which admits two first integrals, $I_1, I_2$, can be realized locally as a Hamilton-Poisson system of the above mentioned type, choosing $(C,H)\in\{(I_1,I_2),(I_2,I_1)\}$ and then, in each case apart, determining the corresponding rescaling function $\nu$ (for details regarding Hamiltonian systems, see e.g. \cite{tudoranPO}, \cite{marsdenratiu}, \cite{cushman}, \cite{ratiu}).

Let us present now the Hamilton-Poisson realization of the system \eqref{sys}, associated with the identification $(C,H)=(I_1, I_2)$. This realization, also mentioned in the Appendix of \cite{tokieda1}, provides the geometric framework for the present approach of the rattleback system. In order to keep unchanged the maximal phase space of the Hamilton-Poisson realization for all positive values of $\lambda$, from now on, we assume for the sake of simplicity that $\lambda>1$ is a natural number.
\begin{proposition}(\cite{tokieda1})\label{t22}
The rattleback system \eqref{sys} admits the following Hamilton-Poisson
realization
\begin{equation}\label{hc}
(\Omega,\nu\Pi_{C},H),
\end{equation}
where, $\Omega:=\{(x,y,z)\in\mathbb{R}^3 \mid y\neq 0\}$, $\nu\in {\mathcal{C}}^\infty(\Omega,\mathbb{R})$, $\nu:=y^{1-\lambda}$, is the rescaling function,
$$\Pi_{C}(x,y,z):=\left[ {\begin{array}{*{20}c}
   0 & z & -y \\
   - z & 0 & x  \\
    y& - x  & 0  \\
\end{array}} \right], ~ \forall (x,y,z)\in\Omega,$$
is the Poisson structure generated by $C:=\dfrac{1}{2}\left(x^2 +y^2 +z^2\right)$, and $H\in {\mathcal{C}}^\infty(\Omega,\mathbb{R})$ given by $H:=xy^\lambda$, is the Hamiltonian function.
\end{proposition}
\noindent \begin{proof} The conclusion follows directly from the following equalities, valid for all $(x,y,z)\in\Omega$:
\begin{align*}
\nu(x,y,z)\Pi_{C}(x,y,z)\cdot\nabla H(x,y,z)&=y^{1-\lambda}\cdot
\left[ {\begin{array}{*{20}c}
   0 & {z} & {-y}  \\
   {-z} & 0 & {x}  \\
   {y} & {-x} & 0  \\
\end{array}} \right]\cdot
\left[ {\begin{array}{*{20}c}
   {y^{\lambda}} \\
   {\lambda x y^{\lambda -1}}  \\
   {0} \\
\end{array}} \right]=
\left[ {\begin{array}{*{20}l}
   {\lambda xz} \\
   {-yz}  \\
   {y^2 -\lambda x^2} \\
\end{array}} \right]\\
&=\left[ {\begin{array}{*{20}r}
   {\dot x} \\
   {\dot y}  \\
   {\dot z} \\
\end{array}} \right].
\end{align*}
\end{proof}
\medskip

Next result provides a whole family of Hamilton-Poisson realizations of the rattleback system \eqref{sys}, parametrized by $SL(2,\mathbb{R})$, such that the Hamilton-Poisson realization given in Proposition \ref{t22} corresponds to the identity matrix $I_2$.
\begin{theorem}
The rattleback system \eqref{sys} admits a family of Hamilton-Poisson realizations parameterized by the group $SL(2,\R)$. More precisely,
$(\Omega,\nu\{\cdot,\cdot\}_{a,b},H_{c,d})$ is a Hamilton-Poisson realization of the dynamics \eqref{sys}, where $\Omega=\{(x,y,z)\in\mathbb{R}^3 \mid y\neq 0\}$, $\nu\in {\mathcal{C}}^\infty(\Omega,\mathbb{R})$, $\nu=y^{1-\lambda}$, is the rescaling function,
$$\{f,g\}_{a,b}:=\langle\nabla C_{a,b},\nabla f\times\nabla g \rangle, ~ \forall f,g\in {\mathcal{C}}^\infty(\Omega,\mathbb{R}),$$
the smooth functions $C_{a,b}, H_{c,d}$ are given by
\begin{align*}
C_{a,b}(x,y,z):&=\dfrac{a}{2}(x^2 +y^2+ z^2)+b xy^{\lambda},\\
H_{c,d}(x,y,z):&=\dfrac{c}{2}(x^2 +y^2 +z^2)+d xy^{\lambda}, ~ \forall (x,y,z)\in\Omega,
\end{align*} and $\left[ {\begin{array}{*{20}c}
   a & b  \\
   c & d  \\
\end{array}} \right]\in SL(2,\mathbb{R}).$
\end{theorem}
\noindent \begin{proof} The conclusion follows directly taking
into account that the matrix formulation of the Poisson structure generated by the bracket
$\{\cdot,\cdot\}_{a,b}$, is given in coordinates by
$$\Pi_{a,b}(x,y,z)=
\left[ {\begin{array}{*{20}c}
   0 & {a z} & -a y -\lambda bxy^{\lambda -1} \\
   {-az} & 0 & {ax+by^{\lambda}}  \\
   {a y +\lambda bxy^{\lambda -1}} & -ax-by^{\lambda} & 0
\end{array}} \right], ~ \forall (x,y,z)\in\Omega.$$
\end{proof}
\medskip

\section{Equilibrium states and their stability analysis}

In this short section we analyze the stability properties of the equilibrium states of the rattleback system \eqref{sys}. The same conclusions were also obtained in \cite{tokieda1}, using a different approach. 

Before stating the result of this section, let us give a characterization of the set of equilibrium states of the rattleback system \eqref{sys}.

\begin{remark}
The equilibrium states of the system \eqref{sys} are the elements of the set
$${\mathcal E}:=\{(M,- M\sqrt{\lambda},0):M\in\mathbb{R}\}\cup\{(M,M\sqrt{\lambda},0):M\in\mathbb{R}\}\cup \{(0,0,M):M\in\mathbb{R}\},$$ illustrated in the Fig. \ref{figech}.
\end{remark}

\begin{figure}[H]
\vspace*{25pt}
\centering
\begin{tabular}{|c|}
\hline
\scalebox{0.9}{\includegraphics*{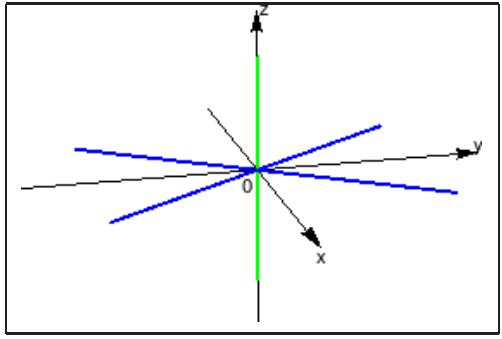}}\\
\hline
\end{tabular}
\caption{\it The equilibrium states of the rattleback system \eqref{sys}.}\label{figech}
\end{figure}

In the following theorem we describe the stability properties of the equilibrium states of the system \eqref{sys}.
\begin{theorem}\label{stability}
The equilibrium states $e^{M}_{\pm}:=(M,\pm M\sqrt{\lambda},0)$, $M\in\mathbb{R}$, are Lyapunov stable, whereas the equilibrium states $e^{M}_3 :=(0,0,M)$, $M\in\mathbb{R}\setminus\{0\}$, are unstable.
\end{theorem}
\noindent \begin{proof}
In order to study the Lyapunov stability of the equilibria $e^{M}_{\pm}$, for $M\neq 0$, we use the Arnold stability criterion \cite{arnold}. Considering the test function $F_\mu:=C-\mu H$, where $\mu\in\mathbb{R}$ is a real parameter, we get that
\begin{itemize}
\item[(i)] ${\rm\bf d}F_\mu(e^{M}_{\pm})=0$ iff $\mu=\mu_{\pm}:=\dfrac{(\pm 1)^{\lambda}}{M^{\lambda -1}\cdot(\sqrt{\lambda})^{\lambda}}$,

 \item[(ii)] $W_{e^{M}_{\pm}}:={\rm  ker}\ {\rm\bf d}H(e^{M}_{\pm})={\rm span}_\mathbb{R}\{(\mp \sqrt{\lambda},1,0),\left(0,0,1\right)\}$,

 \item[(iii)] $\left.{\rm\bf d}^2F_{\mu_{\pm}}(e^{M}_{\pm})\right|_{W_{e^{M}_{\pm}} \times W_{e^{M}_{\pm}}}={\rm diag}\left[2(\lambda +1),1\right]$ is positive definite, since $\lambda>0$.
\end{itemize}

Thus, from the Arnold stability criterion we conclude that for $M\neq 0$, the equilibrium states $e^{M}_{\pm}$ are Lyapunov stable.

For $M=0$, the associated equilibrium state is the origin, which is Lyapunov stable since the Casimir function $C=\dfrac{1}{2}\left(x^2 +y^2 +z^2\right)$ is obviously an associated Lyapunov function.

The instability of the equilibrium states of type $e^{M}_3$, $M\neq 0$, follows directly from the spectral analysis of the linearization of the rattleback system \eqref{sys}. More precisely, for a fixed $M\neq 0$, the spectrum of the linearized dynamics about $e^{M}_3$, consists of the eigenvalues $\{0,-M,\lambda M\}$. As $\lambda >0$ and $M\neq 0$, one of the eigenvalues is strictly positive, and consequently the equilibrium $e^{M}_3$ is unstable.
\end{proof}
\medskip

\section{The local existence of periodic orbits}

In this section we analyze the existence of periodic orbits of the rattleback system \eqref{sys}. In order to do that, let us recall an important consequence of the Hamiltonian nature of the rattleback system \eqref{sys}. More precisely, due to the Hamiltonian nature of the system, its dynamic takes place on the symplectic leaves of the Poisson configuration manifold. Thus, in order to detect the existence of periodic orbits of the rattleback system, a natural approach consists in restricting the system to an arbitrary fixed (regular) symplectic leaf, and then proving the existence of periodic orbits of the restricted system. As the restricted system becomes a symplectic Hamiltonian system, our approach is based on the Lyapunov Center Theorem. Using a different approach, the existence of periodic orbits of the rattleback system was also analyzed in \cite{tokieda1}.

Let us start now our analysis. Using the notations from the previous section, we restrict the rattleback system \eqref{sys} to a (regular) symplectic leaf (of the Poisson manifold $(\Omega,\nu\Pi_{C})$) that contains a Lyapunov stable equilibrium, and then we apply the Lyapunov Center Theorem, in order to get the existence of periodic orbits for the associated dynamics.

More precisely, using Theorem \ref{stability}, we search for periodic orbits located on (regular) symplectic leaves which contain equilibria of type $e^{|M|}_{\pm}:=(|M|,\pm |M|\sqrt{\lambda},0)$, with $M\neq 0$. In order to do that, note that all syplectic leaves of the Poisson manifold $(\Omega,\nu\Pi_{C})$ are regular, and consists of the connected components of $\Omega\cap C^{-1}(\{c\})$, for all $c>0$. Particularly, the symplectic leaves containing $e^{|M|}_{-}$, and respectively $e^{|M|}_{+}$, are given by the open hemispheres
\begin{align*}
\mathcal{O}_{e^{|M|}_{-}}:&=\left\{(x,y,z)\in\Omega \mid x^2 + y^2 + z^2 = (\lambda +1) M^2, ~ y<0 \right\},\\
\mathcal{O}_{e^{|M|}_{+}}:&=\left\{(x,y,z)\in\Omega \mid x^2 + y^2 + z^2 = (\lambda +1) M^2, ~ y>0 \right\}.
\end{align*}
Denoting by ${\mathcal{O}}_{e^{|M|}_{\pm}}:=\mathcal{O}_{e^{|M|}_{-}} \cup \mathcal{O}_{e^{|M|}_{+}}$, the union of the syplectic leaves given above, some straightforward computations imply that the rattleback system \eqref{sys} restricted to ${\mathcal{O}}_{e^{|M|}_{\pm}}$, has the characteristic polynomial associated with the linearization at $e^{|M|}_{\pm}$, given by
$$ p_{e^{|M|}_{\pm}} (\mu ) = -\left[\mu ^2 +2M^2 \lambda (\lambda +1)\right].$$
As $\lambda >0$ and $M\neq 0$, the roots of $ p_{e^{|M|}_{\pm}}$ are $\mu_\pm=\pm i |M| \sqrt{2\lambda (\lambda +1)}\in i\mathbb{R}\setminus\{0\}$. Consequently, applying Lyapunov's Center Theorem for the system \eqref{sys} restricted to $\mathcal{O}_{e^{|M|}_{-}}$, and respectively $\mathcal{O}_{e^{|M|}_{+}}$, we get that around $e^{|M|}_{-}$, and respectively $e^{|M|}_{+}$, the union of the corresponding symplectic leaves, ${\mathcal O}_{e^{|M|}_{\pm}}$, is locally foliated by periodic orbits of the restricted system. Thus, as the symplectic leaves are dynamically invariant, these periodic orbits are also periodic orbits for the system \eqref{sys}. Summarizing, we have proved the following theorem.

\begin{theorem}
Let $e^{|M|}_{\pm}:=(|M|,\pm |M|\sqrt{\lambda},0)$, $M\neq 0$, be a pair of symmetric Lyapunov stable equilibrium states of the rattleback system \eqref{sys}. Then there exist $\varepsilon_{0}>0$ and two one-parameter families $\left\{ \gamma _\varepsilon ^{e^{|M|}_{-}}  \right\}_{0 < \varepsilon  \leq \varepsilon_{0}}$, $\left\{ \gamma _\varepsilon ^{e^{|M|}_{+}}  \right\}_{0 < \varepsilon  \leq \varepsilon_{0}}\subset {\mathcal{O}}_{e^{|M|}_{\pm}}$ of periodic orbits that approaches $e^{|M|}_{-}$, and respectively $e^{|M|}_{+}$, as $\varepsilon\to 0$, with periods
$T_\varepsilon ^{e^{|M|}_{-}}=T_\varepsilon ^{e^{|M|}_{+}} \mathop{\longrightarrow}\limits^{\varepsilon\to 0}\dfrac{\pi\sqrt{2}}{|M|{\sqrt{\lambda(\lambda +1)}}}$. Moreover, the sets
$$\{e^{|M|}_{-} \}  \cup \bigcup\limits_{0 < \varepsilon \leq \varepsilon_{0}  } {\gamma _\varepsilon ^{e^{|M|}_{-}}}, ~~~~ \{e^{|M|}_{+} \}  \cup \bigcup\limits_{0 < \varepsilon \leq \varepsilon_{0}  } {\gamma _\varepsilon ^{e^{|M|}_{+}}},$$
are related through the symmetry about the invariant plane $y=0$, each of them being a smooth two dimensional manifold with boundary ($\gamma _{\varepsilon_{0}}^{e^{|M|}_{-}}$, and respectively $\gamma _{\varepsilon_{0}}^{e^{|M|}_{+}}$), diffeomorphic to the closed unit disk of $\mathbb{R}^2$.
\end{theorem}
\begin{proof}
By Lyapunov's Center Theorem there exist $\varepsilon_{0}>0$ and a one-parameter family $\left\{ \gamma _\varepsilon ^{e^{|M|}_{-}}  \right\}_{0 < \varepsilon  \leq \varepsilon_{0}}\subset \mathcal{O}_{e^{|M|}_{-}}$ of periodic orbits of the rattleback system, that approaches $e^{|M|}_{-}$ as $\varepsilon\to 0$, with periods
$T_\varepsilon ^{e^{|M|}_{-}}\mathop{\longrightarrow}\limits^{\varepsilon\to 0}\dfrac{2\pi}{|M| \sqrt{2\lambda (\lambda +1)}}=\dfrac{\pi\sqrt{2}}{|M|{\sqrt{\lambda(\lambda +1)}}}$. Moreover, the set
$$\{e^{|M|}_{-} \}  \cup \bigcup\limits_{0 < \varepsilon \leq \varepsilon_{0}  } {\gamma _\varepsilon ^{e^{|M|}_{-}}}$$
is a smooth two dimensional manifold with boundary $\gamma _{\varepsilon_{0}}^{e^{|M|}_{-}}$, diffeomorphic to the closed unit disk of $\mathbb{R}^2$.

In order to construct the periodic orbits $\left\{ \gamma _\varepsilon ^{e^{|M|}_{+}}  \right\}_{0 < \varepsilon  \leq \varepsilon_{0}}\subset \mathcal{O}_{e^{|M|}_{+}}$, for each $\varepsilon \in (0, \varepsilon_{0}]$ we denote by $\gamma _\varepsilon ^{e^{|M|}_{+}}$, the symmetry of $\gamma _\varepsilon ^{e^{|M|}_{-}}$ about the invariant plane $y=0$. The rest of the proof is a direct consequence of the fact that $t \mapsto (x(t),y(t),z(t))$ is a ($T-$periodic) solution of the rattleback system \eqref{sys} if and only if $t \mapsto (x(t),-y(t),z(t))$ is a ($T-$periodic) solution.
\end{proof}
\medskip

\section{The image of the energy-Casimir mapping}

The aim of this section is to study the image of the so called energy-Casimir mapping, $\operatorname{Im}(\mathcal{EC})$, naturally associated to the Hamilton-Poisson realization \eqref{hc} of the rattleback system \eqref{sys}. More precisely, we analyze geometric properties of $\operatorname{Im}(\mathcal{EC})$, as well as a semialgebraic splitting of the image. Recall that by semialgebraic splitting of the image, we mean a splitting consisting of semialgebraic manifolds, i.e. manifolds that are described in coordinates by a set of algebraic equalities and inequalities. The results presented here, will be used in the next section, in order to provide a complete topological classification of the orbits of the rattleback system \eqref{sys}.

Let us introduce now the main protagonist of this section, namely, the energy-Casimir mapping, ${\mathcal EC}\in C^\infty(\R^3,\R^2)$,
$${\mathcal EC}(x,y,z):=(H(x,y,z), C(x,y,z)),~ \forall (x,y,z)\in\mathbb{R}^3,$$
where $H,C\in C^\infty(\R^3,\R)$ are given by $H:=xy^{\lambda}$, and $C:=\dfrac{1}{2}(x^2 +y^2 +z^2)$.

Note that the functions $H$ and $C$ given above, are the extensions to $\mathbb{R}^3$ of the Hamiltonian/energy of the system \eqref{hc}, and respectively the Casimir $C$ of the Poisson manifold $(\Omega,\nu\Pi_C)$, as introduced in Theorem \ref{t22}. Recall that even if both the Hamiltonian and the Casimir invariant, are globally defined functions, in order to realize the rattleback system as a Hamilton-Poisson system, a rescaling function was needed, whose maximal domain of definition is $\Omega$, a proper subset of $\mathbb{R}^3$ (given as the complement with respect to $\mathbb{R}^3$ of the invariant plane $y=0$). Summarizing, despite of the terminology, the so called energy-Casimir mapping, is actually related more to the complete integrability of the rattleback system than to the Hamiltonian nature of the dynamics. This is due to the fact that in dimension three, any Hamilton-Poisson system is also completely integrable, whereas the converse implication generally holds true \textit{only} locally (see \cite{tudoranPO}), as is the case here, for the rattleback system.

Next proposition provides an explicit semialgebraic splitting of the image of the energy-Casimir map ${\mathcal EC}$.

\begin{proposition}\label{p51}
 The image of the energy-Casimir map admits the splitting $$Im(\mathcal EC)=S^{-}\cup S^{+},$$ where the subsets $S^{-}, S^{+}\subset\mathbb{R}^2$ are splitting further as a union of semialgebraic manifolds, as follows:
\begin{align*}
 S^{-}&:=\left\{(h,c)\in\mathbb{R}^2:~ h^2 =\lambda^{\lambda}\cdot \left(\dfrac{2}{\lambda +1}\right)^{\lambda +1}\cdot c^{\lambda +1};~ h< 0;~ c>  0\right\}\\
&\bigcup\{(0,0)\}\bigcup\left\{(h,c)\in\mathbb{R}^2:~ h= 0;~ c>0 \right\}\\
&\bigcup\left\{(h,c)\in\mathbb{R}^2:~ h^2 < \lambda^{\lambda}\cdot \left(\dfrac{2}{\lambda +1}\right)^{\lambda +1}\cdot c^{\lambda +1};~ h< 0;~ c> 0\right\},\\
 S^{+}&:=\left\{(h,c)\in\mathbb{R}^2:~ h^2 =\lambda^{\lambda}\cdot \left(\dfrac{2}{\lambda +1}\right)^{\lambda +1}\cdot c^{\lambda +1};~ h> 0;~ c>   0\right\}\\
&\bigcup\{(0,0)\}\bigcup\left\{(h,c)\in\mathbb{R}^2:~ h= 0;~ c>0 \right\}\\
&\bigcup\left\{(h,c)\in\mathbb{R}^2:~ h^2 < \lambda^{\lambda}\cdot \left(\dfrac{2}{\lambda +1}\right)^{\lambda +1}\cdot c^{\lambda +1};~ h> 0;~ c> 0\right\}.
\end{align*}
\end{proposition}
\begin{proof}
 The conclusion follows directly by simple algebraic computation using the definition of the energy-Casimir mapping.
\end{proof}
\medskip

The connection between the semialgebraic splittings of the image $Im({\mathcal EC})$ given by Proposition \ref{p51}, and the equilibrium states of the rattleback system, is given in the following remark. Before stating the remark let us split the set of equilibrium states of the rattleback system according to their stability. 

More precisely, we denote by $$\mathcal{E}^{s}:=\mathcal{E}^{s,\star}\cup\mathcal{E}^{s,0},$$
the set of stable equilibrium states, where
$$\mathcal{E}^{s,\star}:=\{(M,- M\sqrt{\lambda},0),(M, M\sqrt{\lambda},0)\in\mathbb{R}^3:~ M\in\mathbb{R},~ M\neq 0\},~~~ \mathcal{E}^{s,0}:=\{(0,0,0)\},$$
and respectively we denote by $$\mathcal{E}^{u}:=\{(0,0,M)\in\mathbb{R}^3:~ M\in\mathbb{R},~ M\neq 0\},$$
the set of unstable equilibrium states of the rattleback system \eqref{sys}.

\begin{remark}\label{strt}
The semialgebraic canonical Whitney stratifications of the sets $S^{-}, S^{+}$ are described in terms of the image of equilibria of the rattleback system through the map $\mathcal EC$ as follows:

\begin{itemize}
\item[(i)] $S^{-} =\Sigma^{s,-,\star}\cup \Sigma^{s,0} \cup \Sigma^{u}\cup \Sigma^{p,-}$, where
\begin{align*} 
\Sigma^{s,-,\star}&:=\operatorname{Im}(\left.{\mathcal EC}\right|_{{\mathcal E}^{s,\star}})\cap \{(h,c)\in\mathbb{R}^2 :~ h< 0;~ c> 0\},\\
\Sigma^{s,0}&:=\operatorname{Im}(\left.{\mathcal EC}\right|_{{\mathcal E}^{s,0}})=\mathcal EC (0,0,0)= \{(0,0)\},\\
\Sigma^{u}&:=\operatorname{Im}(\left.{\mathcal EC}\right|_{{\mathcal E}^{u}}),\\
\Sigma^{p,-}&:=\left\{(h,c)\in\mathbb{R}^2:~ h^2 < \lambda^{\lambda}\cdot \left(\dfrac{2}{\lambda +1}\right)^{\lambda +1}\cdot c^{\lambda +1};~ h< 0;~ c> 0\right\},
\end{align*}
and $\Sigma^{p,-}$ denotes the principal stratum of $S^{-}$.
\item[(ii)] $S^{+} =\Sigma^{s,+,\star}\cup \Sigma^{s,0} \cup \Sigma^{u}\cup \Sigma^{p,+}$, where
\begin{align*} 
\Sigma^{s,+,\star}&:=\operatorname{Im}(\left.{\mathcal EC}\right|_{{\mathcal E}^{s,\star}})\cap \{(h,c)\in\mathbb{R}^2 :~ h> 0;~ c> 0\},\\
\Sigma^{s,0}&:=\operatorname{Im}(\left.{\mathcal EC}\right|_{{\mathcal E}^{s,0}})=\mathcal EC (0,0,0)= \{(0,0)\},\\
\Sigma^{u}&:=\operatorname{Im}(\left.{\mathcal EC}\right|_{{\mathcal E}^{u}}),\\
\Sigma^{p,+}&:=\left\{(h,c)\in\mathbb{R}^2:~ h^2 < \lambda^{\lambda}\cdot \left(\dfrac{2}{\lambda +1}\right)^{\lambda +1}\cdot c^{\lambda +1};~ h> 0;~ c> 0\right\},
\end{align*}
and $\Sigma^{p,+}$ denotes the principal stratum of $S^{+}$.
\item[(iii)] Moreover, the rattleback system is an example of 3D Hamilton-Poisson system for which the image of the energy-Casimir map is not convexly generated by the image of the stable equilibrium states, i.e.
\begin{align*} 
\operatorname{Im}({\mathcal EC}) \neq \overline{co}\left\{\operatorname{Im}(\left.{\mathcal EC}\right|_{\mathcal{E}^{s}})\right\}.
\end{align*}
\end{itemize}
\end{remark}

All the stratification results can be gathered as shown in Fig. \ref{fig3}.

\begin{figure}[H]
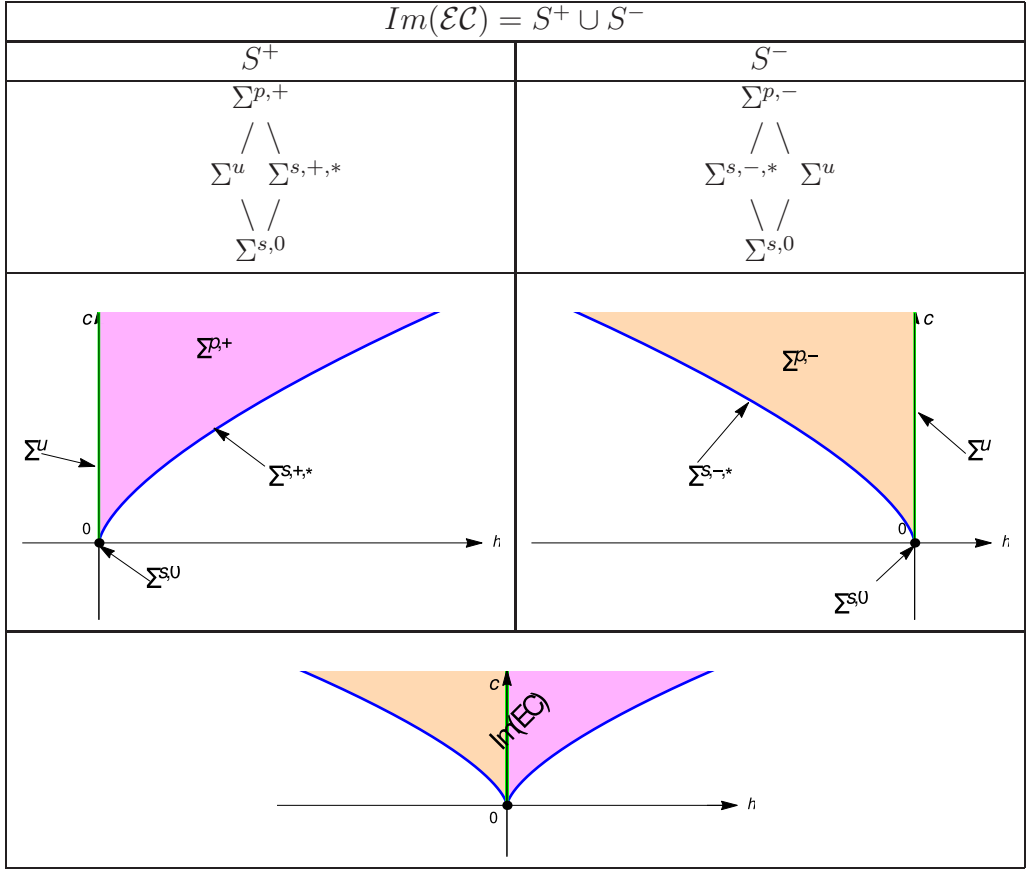

\vspace*{25pt}
\centering
\begin{tabular}{|c|c|c|}
\hline
\multicolumn{2}{|c|}{$Im({\cal EC})=S^+\cup S^-$}\\
\hline
$S^+$ &  $S^-$\\
\hline
$\Sigma^{p,+}$ & $\Sigma^{p,-}$ \\
/ $\backslash$ &/ $\backslash$ \\
$\ \ \ \Sigma^{u}$ $\ \Sigma^{s,+,*}$  & $\Sigma^{s,-,*}$  $\ \Sigma^{u}$ \\
 $\backslash$ / & $\backslash$ / \\
$\Sigma^{s,0}$  & $\Sigma^{s,0}$ \\
\hline
 & \\
\scalebox{0.5}{%
\includegraphics*{imp.%
eps}} &
\scalebox{0.5}{%
\includegraphics*{imm.%
eps}} \\
\hline
\multicolumn{2}{|c|}{}\\
\multicolumn{2}{|c|}{
\scalebox{0.5}{%
\includegraphics*{im.%
eps}}}\\
\hline
\end{tabular}
\caption{\it Semialgebraic splitting of $Im({\mathcal EC})$.}\label{fig3}
\end{figure}

\section{The topology of the fibers of the energy-Casimir mapping}

In this section we describe the topology of the fibers of $\mathcal EC$, considering for our study fibers over regular values of $\mathcal EC$ as well as fibers over the singular values. It will remain an open question how these fibers fit all together in a more abstract fashion, such as bundle structures in the symplectic Arnold-Liouville integrable regular case.

\begin{proposition}\label{topostrt}
According to the stratifications from the previous section, the topology of the fibers of $\mathcal EC$ can be described as in Tables \ref{tab1} and \ref{tab2}:
\vskip1cm
\begin{table}[H]
\centering
\begin{tabular}{|c|c|c|c|c|}
\hline
$S^+$ & \multicolumn{4}{|c|}{ }\\
\hline
$A\subseteq S^+$ & $\Sigma^{p,+}$ & $\Sigma^{s,+,*}$ & $\Sigma^{u}$ & $\Sigma^{s,0}$\\
\hline
 &  &  & \multirow{4}{*}{\scalebox{0.2}{%
\includegraphics*{cercuri.%
eps}}}  & \\
${\cal F}_{(h,c)}\subseteq\R^3,$ & $S^{1} \amalg S^{1}$ & $\{pt\}\times \{pt'\}$ &   & $(0,0,0)$\\
$(h,c)\in A$ &  & $pt\ne pt'$ &  &\\
  &  & & & \\
   &  & & & \\
\hline
Dynamical   & pair of  & pair of stable & heteroclinic & stable\\
description & periodic & equilibrium    & cycles       & equilibrium\\
            & orbits   & states         &              & state\\
\hline
\end{tabular}
\caption{\it Fibers classification corresponding to $S^{+}$.}\label{tab1}
\end{table}

\vskip1cm
\begin{table}[H]
\centering
\begin{tabular}{|c|c|c|c|c|}
\hline
$S^-$ & \multicolumn{4}{|c|}{ }\\
\hline
$A\subseteq S^-$ & $\Sigma^{p,-}$ & $\Sigma^{s,-,*}$ & $\Sigma^{u}$ & $\Sigma^{s,0}$\\
\hline
&  &  & \multirow{4}{*}{\scalebox{0.2}{%
\includegraphics*{cercuri.%
eps}}}  & \\
${\cal F}_{(h,c)}\subseteq\R^3,$ & $S^{1} \amalg S^{1}$ & $\{pt\}\times \{pt'\}$ &   & $(0,0,0)$\\
$(h,c)\in A$ &  & $pt\ne pt'$ & & \\
 &  & & & \\
   &  & & & \\
\hline
Dynamical   & pair of  & pair of stable & heteroclinic & stable\\
description & periodic & equilibrium    & cycles       & equilibrium\\
            & orbits   & states         &              & state\\
\hline
\end{tabular}
\caption{\it Fibers classification corresponding to $S^{-}$.}
\label{tab2}
\end{table}

\end{proposition}
\noindent \begin{proof}
The conclusion follows by simple computations according to the topology of the solution set of the system:
$$\left\{ \begin{array}{l}
 H(x,y,z) = h \\
 C(x,y,z) = c \\
 \end{array} \right.$$
 where $(h,c)$ belongs to the semialgebraic manifolds introduced in the above section.
\end{proof}
\medskip

A presentation that puts together the topological classification of the fibers of $\mathcal EC$ and the topological classification of the symplectic leaves of the Poisson manifold $(\R^3,\nu\Pi_C)$, is given in Fig. \ref{fig6}.

\begin{figure}[H]
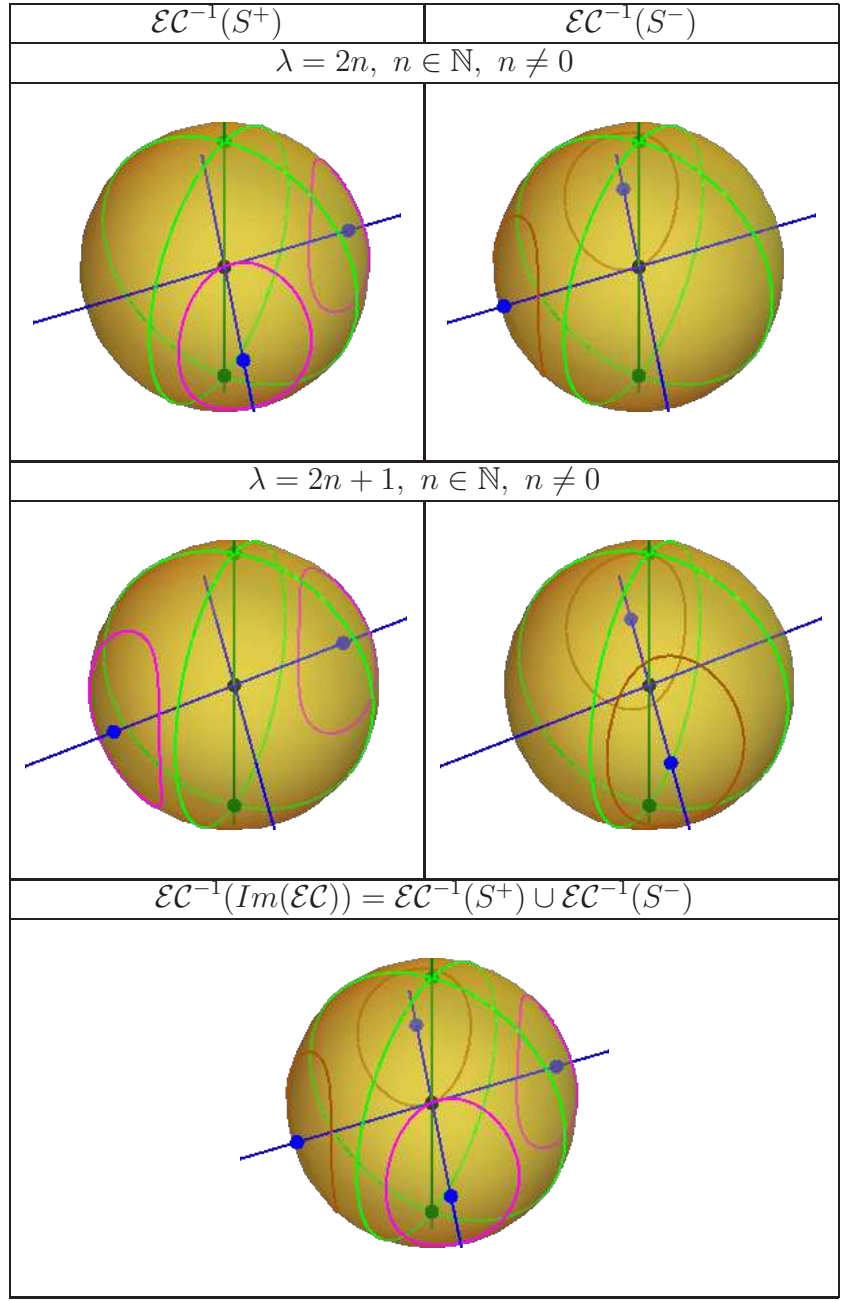

\centering
\vspace*{20pt}
\index{\footnote{}}
\begin{tabular}{|c|c|}
\hline
${\cal EC}^{-1}({S^+})$ &  ${\cal EC}^{-1}(S^-)$\\
\hline
\multicolumn{2}{|c|}{$\lambda=2n,\ n\in\N,\ n\ne0$}\\
\hline
 & \\
\scalebox{0.7}{%
\includegraphics*{sigmap.%
eps}} &
\scalebox{0.7}{%
\includegraphics*{sigmam.%
eps}} \\
 &  \\
\hline
\multicolumn{2}{|c|}{$\lambda=2n+1,\ n\in\N,\ n\ne0$}\\
\hline
 & \\
\scalebox{0.7}{%
\includegraphics*{sigmap_imp.%
eps}} &
\scalebox{0.7}{%
\includegraphics*{sigmam_imp.%
eps}} \\
 &  \\
\hline
\multicolumn{2}{|c|}{${\cal EC}^{-1}(Im({\cal EC}))={\cal EC}^{-1}({S^+})\cup{\cal EC}^{-1}({S^-})$}\\
\hline
\multicolumn{2}{|c|}{ }\\
\multicolumn{2}{|c|}{
\scalebox{0.7}{%
\includegraphics*{sigma.%
eps}}}\\
 \multicolumn{2}{|c|}{ }\\
\hline
\end{tabular}
\caption{\it Phase portrait splitting.}\label{fig6}
\end{figure}

\section{Heteroclinic solutions}

In this section, using the topology of the fibers of $\mathcal{EC}$, we obtain explicit formulas for the heteroclinic solutions of the system \eqref{sys}. The associated orbits connect pairs of unstable equilibrium states of the type $\{(0,0,-|M|),(0,0,|M|)\}$ for $M\in\mathbb{R}\setminus\{0\}$. 
Recall from the previous section that the existence of heteroclinic orbits corresponds in the image of the energy-Casimir mapping to points in the strata $\Sigma^u$.
Using an implicit formulation, these orbits are described by the set of solutions of the following semialgebraic system
\begin{equation}\label{het}
\left\{ \begin{array}{l}
\mathcal{EC}(x,y,z) = (h,c), \\ 
(h,c)\in\Sigma^u,\\
(h,c)=\mathcal{EC}(0,0,\pm |M|),
\end{array} \right.
\end{equation}
with $M\ne0$.
In order to obtain an explicit form of the heteroclinic solutions, we reduce the system \eqref{sys} from dimension three to dimension one, applying the constraints \eqref{het}, and then we integrate the resulting differential equation. Reconstructing the solutions of the initial system, we get the explicit form of the heteroclinic solutions connecting the unstable equilibrium states $(0,0,-|M|)$ and $(0,0,|M|)$ with $M\ne0$.
In order to do that, we adopt the following notations:
$$\left\{ \begin{array}{l}
x(t):= \dfrac{2 M^2 \left( \cosh [|M|(\lambda t+ k)]+ \sinh [|M|(\lambda t+ k)] \right)}{M^2 \cosh(2|M|k)+ \cosh(2|M|\lambda t)+ M^2 \sinh(2|M|k)+\sinh(2|M|\lambda t)}, \\
y(t):=\dfrac{2 M^2}{(M^2 +1)\cosh[|M|(t+k)]+(M^2 -1)\sinh[|M|(t+k)]},\\
z(t):=\dfrac{|M|(-1+M^2 \cosh[2|M|(t+k)] +M^2 \sinh[2|M|(t+k)])}{1+M^2 \cosh[2|M|(t+k)]+M^2 \sinh[2|M|(t+k)]},\\
\tilde{z}(t):=\dfrac{|M|(M^2 \cosh(2|M|k)-\cosh(2|M|\lambda t)+M^2 \sinh(2|M|k)-\sinh(2|M|\lambda t))}{M^2 \cosh(2|M|k)+\cosh(2|M|\lambda t)+M^2 \sinh(2|M|k)+\sinh(2|M|\lambda t)},
\end{array} \right. t,k\in\R.$$
Summarizing, for each $M\neq 0$, we get the existence of four heteroclinic solutions of the rattleback system \eqref{sys} connecting the unstable equilibria $e_{3}^{-|M|}=(0,0,-|M|)$ and $e_{3}^{|M|}=(0,0,|M|)$, given by
\begin{align*}
\mathcal{H}_{(+,0)}^{(0,0,\pm |M|)}(t)&=:\left(x(t),0,\tilde{z}(t)\right),\\
\mathcal{H}_{(-,0)}^{(0,0,\pm |M|)}(t)&=:\left(-x(t),0,\tilde{z}(t)\right),\\
\mathcal{H}_{(0,+)}^{(0,0,\pm |M|)}(t)&=:\left(0,y(t),z(t)\right),\\
\mathcal{H}_{(0,-)}^{(0,0,\pm |M|)}(t)&=:\left(0,-y(t),z(t)\right), ~~ t\in\mathbb{R}.
\end{align*}
\begin{remark}
The orbits of the above defined heteroclinic solutions, connecting the unstable equilibria $e_{3}^{-|M|},e_{3}^{|M|}$, are in one-to-one correspondence with the open semicircles obtained from the intersection between the invariant twice punctured sphere $C^{-1}(\{M^2/2\})\setminus\{e_{3}^{-|M|},e_{3}^{|M|}\}$ and the union of invariant planes $xy=0$.
\end{remark}

\section{Lax Formulation}

The aim of this short section is to provide an explicit Lax formulation of the rattleback system \eqref{sys}.

Before stating the result of this section, note that the system \eqref{sys} restricted to a regular symplectic leaf, gives rise to a symplectic Hamiltonian system that is completely integrable in the sense of Liouville, and consequently it admits a Lax formulation.

The following proposition offers a positive answer regarding the existence of a Lax formulation for the unrestricted system.
\begin{proposition}
The rattleback system \eqref{sys} can be written in the Lax form $\dot L=[L,B]$, where $[L,B]:=L B- B L$, and the matrices $L$ and $B$ are given by
$$
L= \left[
\begin{array}{ccc}
 0 & -x + y\sqrt{\lambda}  & x\sqrt{\lambda} +y  \\
 x - y\sqrt{\lambda} & 0 &  z\sqrt{\lambda +1} \\
 -x\sqrt{\lambda} - y & - z\sqrt{\lambda +1} & 0 
\end{array}
\right],$$
$$B=\left[
\begin{array}{ccc}
 0 & (-x\sqrt{\lambda} + y )\sqrt{\lambda +1} & 0  \\
 (x\sqrt{\lambda} - y )\sqrt{\lambda +1} & 0 & z\sqrt{\lambda} \\
 0 & - z\sqrt{\lambda} & 0 
\end{array}
\right].$$
\end{proposition}

\section{Asymptotic stabilization of an arbitrary fixed orbit/cycle of the rattleback dynamical system}

In this section we provide a method to stabilize asymptotically any arbitrary fixed orbit/cycle of the rattleback dynamical system. More precisely, for an a-priori fixed orbit/cycle $\mathcal{O}$ of the rattleback system \eqref{sys} we construct explicitly a conservative perturbation (in the sense that the Casimir invariant remains a first integral of the perturbed system) which preserves the orbit/cycle $\mathcal{O}$, keeps its dynamical nature unchanged, and moreover, the orbit/cycle $\mathcal{O}$ becomes asymptotically stable (with respect to perturbations along the invariant manifold $C^{-1}(\{C(\mathcal{O})\})$) as an orbit of the perturbed dynamics, regardless of its initial stability nature. Additionally, the same perturbation provides asymptotic stabilization of all the dynamical elements from the set $(\mathcal{EC})^{-1}(\{\mathcal{EC}(\mathcal{O})\})$. As the perturbed dynamics is conservative ($C$ remains a first integral of the perturbed system too), in order to approach asymptotically some orbit/cycle $\mathcal{O}$, we must start from a point located on the invariant manifold $C^{-1}(\{C(\mathcal{O})\})$. \textbf{Hence, from now on, asymptotic stabilization of an orbit/cycle $\mathcal{O}$ means actually asymptotic stabilization with respect to perturbations along $C^{-1}(\{C(\mathcal{O})\})$.}

The explicit construction of the perturbations agrees with the topological classification of the orbits given in Section 6. More precisely, for each arbitrary fixed dynamical object described in Proposition \ref{topostrt} (i.e. pair of stable equilibrium points, pair of periodic orbits, and heteroclinic cycles) we construct an explicit conservative perturbation of the rattleback system with the above mentioned properties. Moreover, in the case of periodic orbits, we construct conservative perturbations leading to a stronger type of asymptotic stability, i.e. the so called asymptotic stability with phase.

The general perturbation procedure we use, works as follows. First, pick a dynamical element located on some fiber $(\mathcal{EC})^{-1}(\{(h,c)\})$ of the energy-Casimir mapping, as classified in Proposition \ref{topostrt}. Then perturb the rattleback system \eqref{sys} by adding a smooth vector field $X_0$ such that the resulting vector field admits $C$ as first integral and keeps dynamically invariant the fiber $(\mathcal{EC})^{-1}(\{(h,c)\})=\{H=h\}\cap \{C=c\}$, i.e. the dynamical element(s) to be stabilized. Following \cite{tudoran1}, the vector field $X_0$ is proportional (at least in some open subset) to the vector field 
\begin{align}\label{X0}
\begin{split}
(H(x,y,z)-h)\cdot &[\nabla C(x,y,z) \times (\nabla C(x,y,z) \times \nabla H(x,y,z))]\\
&=(H(x,y,z)-h)\cdot y^{\lambda-1} \cdot\left[ \begin{array}{c} y(\lambda x^2 -y^2 -z^2) \\ x(-\lambda x^2 +y^2 -\lambda z^2) \\ (\lambda +1)xyz \end{array} \right].
\end{split}
\end{align}
The proportionality function will be constructed separately for each case, in accordance with the Lyapunov function used in order to asymptotically stabilize the chosen dynamical element.

\subsection{Asymptotic stabilization of the Lyapunov stable equilibrium states of the rattleback system}

The aim of this section is to provide a method to stabilize asymptotically each arbitrary fixed Lyapunov stable equilibrium state of the rattleback system \eqref{sys}, and in the same time to keep the conservative nature of the dynamics, in the sense that the Casimir invariant, $C = \dfrac{1}{2} (x^2 +y^2 +z^2)$, remains a first integral of the perturbed system. Thus, as already mentioned in the previous section, the asymptotic stabilization procedure makes sense only for perturbations along the corresponding level sets of $C$. Note that the only Lyapunov stable equilibrium state which cannot be asymptotically stabilized using this approach is the origin, as it belongs to the degenerate sphere $C^{-1}(\{C(0,0,0)\})=\{(0,0,0)\}$.

Before starting the construction of the perturbed dynamical system, let us recall from Remark \ref{strt} that the image through the energy-Casimir mapping of the nontrivial Lyapunov stable equilibrium states of the rattleback system, is given by  $\Sigma^{s,-,\star}\cup\Sigma^{s,+,\star}$. Moreover, according to Proposition \ref{topostrt}, each point, $(h,c)\in\Sigma^{s,-,\star}\cup\Sigma^{s,+,\star}$, corresponds to a pair of Lyapunov stable equilibrium states of the rattleback system \eqref{sys}, given by $\mathcal{EC}^{-1}(\{(h,c)\})$. Reacall from Theorem \ref{stability} that each nontrivial Lyapunov stable equilibrium state is given by $(M,-M\sqrt{\lambda},0)$ or $(M,M\sqrt{\lambda},0)$ for some $M\in\mathbb{R}\setminus\{0\}$.

Consequently, on each Casimir level set, $C^{-1}(\{(\lambda+1)M^2 /2\})$, $M\neq 0$, there exist exactly four Lyapunov stable equilibrium states of the rattleback system \eqref{sys}, i.e. $-e_{-}^{|M|}$, $-e_{+}^{|M|}$, $e_{-}^{|M|}$, and $e_{+}^{|M|}$, where 
$e_{-}^{|M|}:=(|M|,-|M|\sqrt{\lambda},0),~e_{+}^{|M|}:=(|M|,|M|\sqrt{\lambda},0)$.

As the geometric location of these equilibrium states depends explicitly on the parity of the parameter $\lambda$, we distinguish two cases according to the parity of $\lambda$. More precisely, we construct two perturbations of the rattleback system \eqref{sys}, denoted by $p_{\pm}$ (depending on $\lambda$ and $M$), and we prove that for $\lambda=2n$, $n\in\mathbb{N}\setminus \{0\}$, the perturbation $p_{-}$ stabilize asymptotically the equilibrium states $e_{-}^{|M|}$ and $e_{+}^{|M|}$, and the perturbation $p_{+}$ stabilize asymptotically the equilibrium states $-e_{-}^{|M|}$ and $-e_{+}^{|M|}$, while for $\lambda=2n+1$, $n\in\mathbb{N}\setminus \{0\}$, the perturbation $p_{-}$ stabilize asymptotically the equilibrium states $e_{+}^{|M|}$ and $-e_{+}^{|M|}$, and the perturbation $p_{+}$ stabilize asymptotically the equilibrium states $e_{-}^{|M|}$ and $-e_{-}^{|M|}$. 

As the perturbations we are looking for must be conservative (in the sense that the Casimir invariant, $C = 1/2 \cdot (x^2 +y^2 +z^2)$, remains a first integral for the perturbed dynamics), the asymptotic stabilization procedure is supposed to take place only on the dynamically invariant sphere $C^{-1}(\{(\lambda+1)M^2 /2\})$.

In order to construct the perturbations $p_{\pm}$ we take into account that for $\lambda=2n$, $n\in\mathbb{N}\setminus \{0\}$, the pair of equilibrium states $\{e_{-}^{|M|},e_{+}^{|M|}\}$ is given by the intersection between $C^{-1}(\{(\lambda+1)M^2 /2\})$ and $H^{-1}(\{|M|^{\lambda+1}(\sqrt{\lambda})^{\lambda}\})$, and similarly, the pair of equilibrium states $\{-e_{-}^{|M|},-e_{+}^{|M|}\}$ is given by the intersection between $C^{-1}(\{(\lambda+1)M^2 /2\})$ and $H^{-1}(\{-|M|^{\lambda+1}(\sqrt{\lambda})^{\lambda}\})$. On the other hand, for $\lambda=2n+1$, $n\in\mathbb{N}\setminus \{0\}$, the pair of equilibrium states $\{e_{+}^{|M|},-e_{+}^{|M|}\}$ is given by the intersection between $C^{-1}(\{(\lambda+1)M^2 /2\})$ and $H^{-1}(\{|M|^{\lambda+1}(\sqrt{\lambda})^{\lambda}\})$, while the pair of equilibrium states $\{e_{-}^{|M|},-e_{-}^{|M|}\}$ is given by the intersection between $C^{-1}(\{(\lambda+1)M^2 /2\})$ and $H^{-1}(\{-|M|^{\lambda+1}(\sqrt{\lambda})^{\lambda}\})$. 

Let us state now the main result of this section.

\begin{theorem}
Let $\varepsilon >0$ be a real parameter and let $M\in\mathbb{R}\setminus\{0\}$. The following perturbations of the rattleback system \eqref{sys} 
\begin{equation*}
p_{\pm}: \left[ \begin{array}{c} \dot{x} \\ \dot{y} \\ \dot{z} \end{array} \right]=\left[ \begin{array}{c}  \lambda xz \\ -yz \\ y^2 -\lambda x^2 \end{array} \right]+\varepsilon \left( xy^{\lambda}\pm|M|^{\lambda +1}(\sqrt{\lambda})^{\lambda}\right)y^{\lambda -1}\left[ \begin{array}{c} y(\lambda x^2 -y^2 -z^2) \\ x(-\lambda x^2 +y^2 -\lambda z^2) \\ (\lambda +1)xyz \end{array} \right],
\end{equation*}
$(x,y,z)\in\mathbb{R}^{3}$, asymptotically stabilize the Lyapunov stable equilibrium states $\pm e_{-}^{|M|}, \pm e_{+}^{|M|}$ of the rattleback system \eqref{sys}, with respect to perturbations along the invariant manifold $C^{-1}(\{(\lambda+1)M^2 /2\})$.

More precisely, if $\lambda\in 2\mathbb{N}\setminus\{0\}$ then the perturbation $p_{-}$ asymptotically stabilize the pair of equilibrium states $\{e_{-}^{|M|}, e_{+}^{|M|}\}$, whereas the perturbation $p_{+}$ asymptotically stabilize the pair of equilibrium states $\{-e_{-}^{|M|}, -e_{+}^{|M|}\}$. On the other hand, if $\lambda\in 2\mathbb{N}+1 \setminus\{1\}$ then the perturbation $p_{-}$ asymptotically stabilize the pair of equilibrium states $\{-e_{+}^{|M|}, e_{+}^{|M|}\}$, whereas the perturbation $p_{+}$ asymptotically stabilize the pair of equilibrium states $\{-e_{-}^{|M|}, e_{-}^{|M|}\}$.
\end{theorem}
\begin{proof}
First of all, note that (by construction, \eqref{X0}) the equilibrium states $\pm e_{-}^{|M|}, \pm e_{+}^{|M|}$ of the rattleback system \eqref{sys}, are also equilibrium states of the perturbed systems $p_{\pm}$.

The rest of the proof is based on a version of LaSalle's stability criterion (introduced in \cite{metripl}) which states that given a smooth vector field $X\in\mathfrak{X}(\mathbb{R}^n)$ and an equilibrium state $x_0$, if there exists a $\mathcal{C}^1$ real function $L_{x_0}$ defined on some compact neighborhood $U_{x_0}$ of $x_e$ such that $L_{x_0}(x_0)=0$, $L_{x_0}(x)>0$, for all $x\in U_{x_0}\setminus\{x_0\}$, and $\mathcal{L}_{X}{L_{x_0}}(x)\leq 0$, for all $x\in U_{x_0}$, then there exists $V_{x_0}\subset U_{x_0}$ an open neighborhood of $x_0$, such that $\omega(x)\subseteq \mathcal{M}_{x_0}$ for all $x\in V_{x_0}$, where $\mathcal{M}_{x_0}$ is the largest dynamically invariant subset of $\{x\in U_{x_0}:  \mathcal{L}_X L_{x_0}(x)=0\}$; here the notation $\mathcal{L}_X$ stands for the Lie derivative along the vector field $X$.

Using this version of LaSalle's stability criterion, we prove now the conclusion for $\lambda=2n$, $n\in\mathbb{N}\setminus \{0\}$. (The proof of the case when $\lambda$ is odd, follows mimetically.)

\textit{In order to do that, let us show that the perturbation $p_{-}$ asymptotically stabilizes both equilibrium states $e_{+}^{|M|}$ and $e_{-}^{|M|}$.} 

We first prove that $p_{-}$ asymptotically stabilizes the equilibrium $e_{+}^{|M|}$. Let $U_{e_{+}^{|M|}}$ be a compact neighborhood of $e_{+}^{|M|}$ such that $-e_{+}^{|M|}, -e_{-}^{|M|}, e_{-}^{|M|}\notin U_{e_{+}^{|M|}}$ and $U_{e_{+}^{|M|}}\cap\{(x,y,z)\in\mathbb{R}^3: y=0\}=\emptyset$. Let $L_{e_{+}^{|M|}}:U_{e_{+}^{|M|}}\rightarrow \mathbb{R}$ be a smooth function, given by
$$
L_{e_{+}^{|M|}}(x,y,z):=\left[\dfrac{1}{2}(x^2 +y^2 +z^2) - \dfrac{1}{2}(\lambda +1)M^2 \right]^2 +\left[xy^{\lambda} - |M|^{\lambda +1}(\sqrt{\lambda})^{\lambda}\right]^2,
$$
for all $(x,y,z)\in U_{e_{+}^{|M|}}$.

Note that $L_{e_{+}^{|M|}}(x,y,z) =0$ if and only if $$(x,y,z)\in C^{-1}(\{(\lambda+1)M^2 /2\}) \cap H^{-1}(\{|M|^{\lambda+1}(\sqrt{\lambda})^{\lambda}\})\cap U_{e_{+}^{|M|}}.$$ As $C^{-1}(\{(\lambda+1)M^2 /2\}) \cap H^{-1}(\{|M|^{\lambda+1}(\sqrt{\lambda})^{\lambda}\})= \{e_{-}^{|M|},e_{+}^{|M|}\}$ and by definition $e_{-}^{|M|}\notin U_{e_{+}^{|M|}}$ it follows that $L_{e_{+}^{|M|}}(e_{+}^{|M|}) =0$ and moreover $L_{e_{+}^{|M|}}(x,y,z) > 0$, for all $(x,y,z)\in  U_{e_{+}^{|M|}}\setminus \{e_{+}^{|M|}\}$. Consequently, we have checked the first two hypotheses of the above mentioned version of LaSalle's stability criterion.

In order to check the third hypothesis, let us denote by $X_{p_{-}}$ the vector field which generates the perturbed system $p_{-}$. Then, some straightforward computations lead to the following equality, valid for all $(x,y,z)\in U_{e_{+}^{|M|}}$:
\begin{equation*}
\mathcal{L}_{X_{p_{-}}} L_{e_{+}^{|M|}} (x,y,z)=-2\varepsilon y^{2(\lambda -1)}\left[ xy^{\lambda}-|M|^{\lambda +1}(\sqrt{\lambda})^{\lambda}\right]^{2} \left[ \lambda^2 x^2 z^2 + y^2 z^2 +(y^2 -\lambda x^2)^2\right],
\end{equation*}
which implies that $\mathcal{L}_{X_{p_{-}}} L_{e_{+}^{|M|}} (x,y,z)\leq 0,~ \forall (x,y,z)\in U_{e_{+}^{|M|}}$, and hence the last hypothesis is also  verified.

Note that $\mathcal{L}_{X_{p_{-}}} L_{e_{+}^{|M|}} (x,y,z)=0$ if and only if 
$$(x,y,z)\in U_{e_{+}^{|M|}}\cap \left(\{(x,y,z)\in\mathbb{R}^3: y=0\}\cup H^{-1}(\{|M|^{\lambda+1}(\sqrt{\lambda})^{\lambda}\})\cup \mathcal{E}\right),$$
where $\mathcal{E}$ stands for the set of equilibrium states of the rattleback system \eqref{sys}.

From LaSalle's stability criterion and the fact that $C$ is a first integral of the perturbed system $p_{-}$, we get the existence of an open neighborhood of $e_{+}^{|M|}$, $V_{e_{+}^{|M|}}\subset U_{e_{+}^{|M|}}$, such that for every $(x_0 ,y_0 ,z_0)\in V_{e_{+}^{|M|}}\cap C^{-1}(\{(\lambda+1)M^2 /2\})$, the solution $(x(t),y(t),z(t))$ of the perturbed system $p_{-}$, starting from $(x_0 ,y_0 ,z_0)$ at $t=0$, approaches (as $t\rightarrow \infty$) the largest invariant subset of
$$
V_{e_{+}^{|M|}}\cap C^{-1}(\{(\lambda+1)M^2 /2\})\cap \left(\{(x,y,z)\in\mathbb{R}^3: y=0\}\cup H^{-1}(\{|M|^{\lambda+1}(\sqrt{\lambda})^{\lambda}\})\cup \mathcal{E}\right).
$$
As
$$V_{e_{+}^{|M|}}\cap C^{-1}(\{(\lambda+1)M^2 /2\})\cap \left(\{(x,y,z)\in\mathbb{R}^3: y=0\}\cup H^{-1}(\{|M|^{\lambda+1}(\sqrt{\lambda})^{\lambda}\})\cup \mathcal{E}\right)=\{e_{+}^{|M|}\}
$$
it follows that $e_{+}^{|M|}$ is an asymptotically stable equilibrium point of $p_{-}$ with respect to perturbations along the invariant manifold $C^{-1}(\{C(e_{+}^{|M|})\})=C^{-1}(\{(\lambda+1)M^2 /2\})$.
\newline

In order to prove that $e_{-}^{|M|}$ is also an asymptotically stable equilibrium point of $p_{-}$ with respect to perturbations along the invariant manifold $C^{-1}(\{C(e_{-}^{|M|})\})=C^{-1}(\{(\lambda+1)M^2 /2\})$, we use the same approach, the only difference being that this time we choose $U_{e_{-}^{|M|}}$ to be a compact neighborhood of $e_{-}^{|M|}$ such that $e_{+}^{|M|}, -e_{+}^{|M|}, -e_{-}^{|M|}\notin U_{e_{-}^{|M|}}$ and $U_{e_{-}^{|M|}}\cap\{(x,y,z)\in\mathbb{R}^3: y=0\}=\emptyset$. Next, we define a smooth function $L_{e_{-}^{|M|}}:U_{e_{-}^{|M|}}\rightarrow \mathbb{R}$, given by the same formula as the function $L_{e_{+}^{|M|}}$, i.e. 
$$
L_{e_{-}^{|M|}}(x,y,z)=\left[\dfrac{1}{2}(x^2 +y^2 +z^2) - \dfrac{1}{2}(\lambda +1)M^2 \right]^2 +\left[xy^{\lambda} - |M|^{\lambda +1}(\sqrt{\lambda})^{\lambda}\right]^2,
$$
for all $(x,y,z)\in U_{e_{-}^{|M|}}$. The rest of the proof follows mimetically the proof that $e_{+}^{|M|}$ is an asymptotically stable equilibrium point of $p_{-}$ with respect to perturbations along the invariant manifold $C^{-1}(\{C(e_{+}^{|M|})\})=C^{-1}(\{(\lambda+1)M^2 /2\})$.
\newline
\newline

\textit{Now, in order to prove that the perturbation $p_{+}$ asymptotically stabilizes both equilibrium states $-e_{-}^{|M|}$ and $-e_{+}^{|M|}$, we use a similar approach, the main difference being the formula of the Lyapunov function, i.e. this time we use a smooth function given by
$$
(x,y,z)\mapsto \left[\dfrac{1}{2}(x^2 +y^2 +z^2) - \dfrac{1}{2}(\lambda +1)M^2 \right]^2 +\left[xy^{\lambda} + |M|^{\lambda +1}(\sqrt{\lambda})^{\lambda}\right]^2.
$$}
\end{proof}

The phase portraits of the perturbed systems $p_{\pm}$ in accordance with the geometric location of the equilibrium states $- e_{-}^{|M|}, e_{-}^{|M|}, - e_{+}^{|M|}, e_{+}^{|M|}$, are illustrated in Fig. \ref{tabpm}.

\begin{figure}[H]
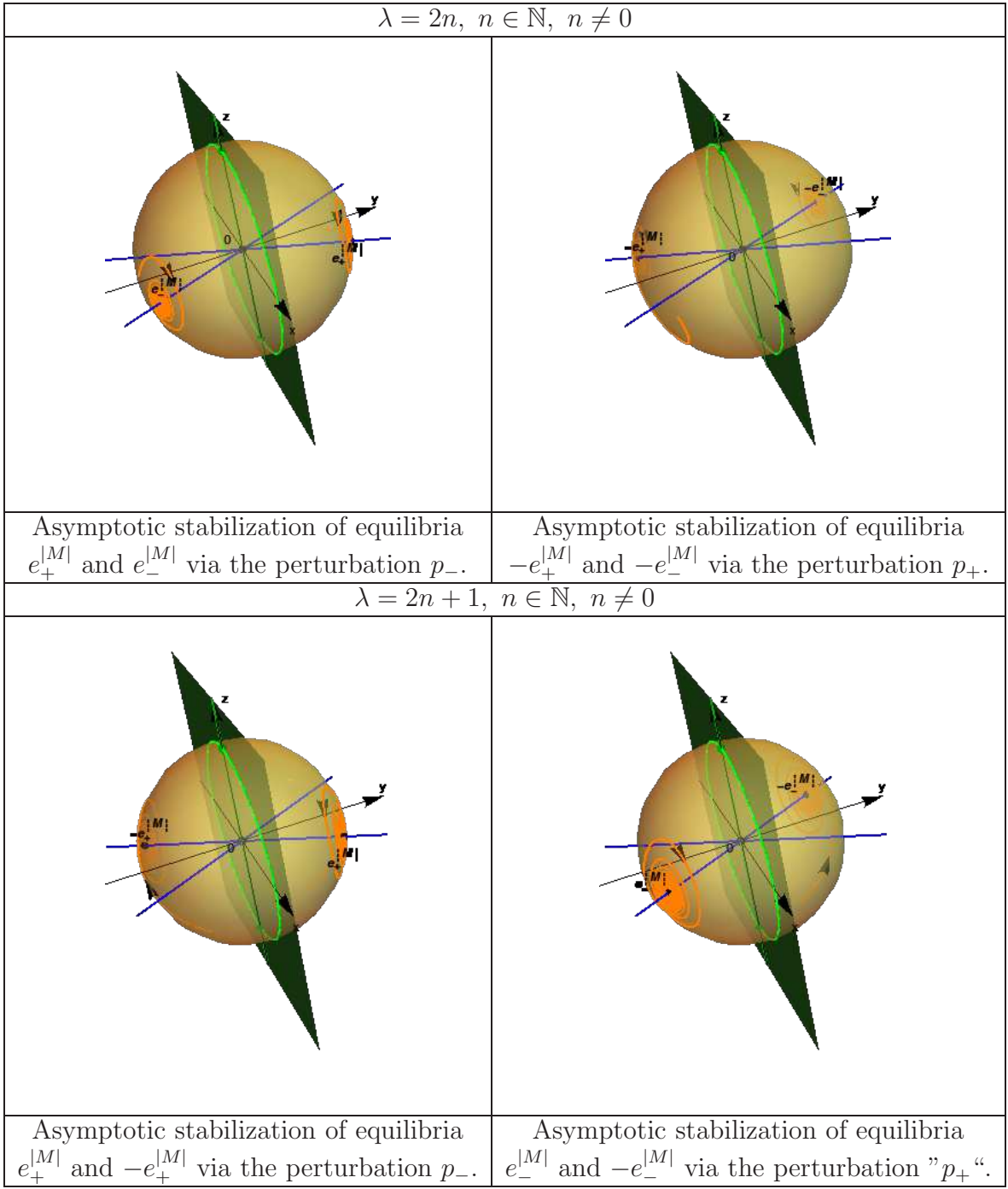

\centering
\vspace*{20pt}
\index{\footnote{}}
\begin{tabular}{|c|c|}
\hline
\multicolumn{2}{|c|}{$\lambda=2n,\ n\in\N,\ n\ne0$}\\
\hline
 & \\
\scalebox{0.7}{%
\includegraphics*{pert_ech_par1.%
eps}} &
\scalebox{0.7}{%
\includegraphics*{pert_ech_par2.%
eps}} \\
 &  \\
\hline
Asymptotic stabilization of equilibria & Asymptotic stabilization of equilibria\\
$e_{+}^{|M|}$ and $e_{-}^{|M|}$ via the perturbation $p_-$. & $-e_{+}^{|M|}$ and $-e_{-}^{|M|}$ via the perturbation $p_+$.\\
\hline
\multicolumn{2}{|c|}{$\lambda=2n+1,\ n\in\N,\ n\ne0$}\\
\hline
 & \\
\scalebox{0.7}{%
\includegraphics*{pert_ech_impar1.%
eps}} &
\scalebox{0.7}{%
\includegraphics*{pert_ech_impar2.%
eps}} \\
 &  \\
\hline
Asymptotic stabilization of equilibria & Asymptotic stabilization of equilibria\\
$e_{+}^{|M|}$ and $-e_{+}^{|M|}$ via the perturbation $p_-$. & $e_{-}^{|M|}$ and $-e_{-}^{|M|}$ via the perturbation "$p_+$``.\\
\hline
\end{tabular}
\caption{\it Phase portraits of the perturbed systems $p_{\pm}$.}\label{tabpm}
\end{figure}

\subsection{Asymptotic stabilization with phase of the periodic orbits of the rattleback system}

In this section we provide a method to stabilize asymptotically with phase, each arbitrary fixed periodic orbit of the rattleback system \eqref{sys}, and in the same time to keep the conservative nature of the dynamics, in the sense that the Casimir invariant $C = \dfrac{1}{2} (x^2 +y^2 +z^2)$ remains a first integral of the perturbed system. Thus, as already mentioned before, the asymptotic stabilization procedure makes sense only for perturbations along the corresponding level sets of $C$.

Before stating the main result of this section, let us recall from Proposition \ref{topostrt} that each point $(h,c)\in\Sigma^{p,-}\cup\Sigma^{p,+}$ corresponds to a pair of periodic orbits of the rattleback system \eqref{sys}, given by $\mathcal{EC}^{-1}(\{(h,c)\})$.

\begin{theorem}\label{spor}
Let $\varepsilon >0$ be a real parameter, and let $(h,c)\in \Sigma^{p,-}\cup\Sigma^{p,+}$. The following perturbation of the rattleback system \eqref{sys} 
\begin{equation*}
p: \left[ \begin{array}{c} \dot{x} \\ \dot{y} \\ \dot{z} \end{array} \right]=\left[ \begin{array}{c}  \lambda xz \\ -yz \\ y^2 -\lambda x^2 \end{array} \right]+\varepsilon \left( xy^{\lambda} - h \right)y^{\lambda -1}\left[ \begin{array}{c} y(\lambda x^2 -y^2 -z^2) \\ x(-\lambda x^2 +y^2 -\lambda z^2) \\ (\lambda +1)xyz \end{array} \right],
\end{equation*}
$(x,y,z)\in V:=\mathbb{R}^{3}\setminus\{\{(x,y,z)\in\mathbb{R}^{3}: y = 0\}\cup\{(M,- M\sqrt{\lambda},0):M\in\mathbb{R}\}\cup\{(M,M\sqrt{\lambda},0):M\in\mathbb{R}\}\cup \{(0,0,M):M\in\mathbb{R}\}\}$, asymptotically stabilizes with phase both periodic orbits, $\mathcal{EC}^{-1}(\{(h,c)\})$, of the rattleback system \eqref{sys}, with respect to perturbations in $V$ along the invariant manifold $C^{-1}(\{c\})$. 
\end{theorem}
\begin{proof}
First of all, note that (by construction, \eqref{X0}) both periodic orbits, $\mathcal{EC}^{-1}(\{(h,c)\})$, of the rattleback system \eqref{sys}, are also periodic orbits of the perturbed system $p$.

The rest of the proof follows directly from a stability criterion given in \cite{tudoranPO} that concerns Hamiltonian dynamical systems of the type
\begin{equation}\label{EPR}
\dot{u}=\nu(u)\left(\nabla H(u)\times \nabla C(u)\right),
\end{equation}
defined on the Poisson manifold $(\Omega,\nu\Pi_{C})$, where $\Omega\subseteq \mathbb{R}^3$ is an open set, and $H,C,\nu \in C^{\infty}(\Omega,\mathbb{R})$ are smooth real functions, such that $H$ and $C$ are functionally independent on a non-empty open subset $V\subseteq \Omega$. 

More precisely, the criterion works as follows. Assume there exists $\Gamma\subset V$ is a periodic orbit of \eqref{EPR}. If $\Gamma \subseteq (H,C)^{-1}(\{(h,c)\})$, where $(h,c)\in\mathbb{R}^2$ is a regular value for the map $(H,C):\Omega\rightarrow \mathbb{R}^2$, then the following conclusion holds true.

If $c$ is a regular value of the map $C:\Omega\rightarrow \mathbb{R}$, then for every smooth function $\alpha \in C^{\infty}(V,(0,\infty))$, 
$\Gamma$, as a periodic orbit of the perturbed dynamical system
\begin{align*}
\dfrac{\mathrm{d}u}{\mathrm{d}t} =\nu(u)\left(\nabla H(u)\times \nabla C(u)\right)-\alpha(u)(H(u)-h)\left[\nabla C(u)\times \left(\nabla H(u)\times \nabla C(u)\right)\right],
\end{align*}
$u\in V$, is orbitally asymptotically stable with phase, with respect to perturbations in $V$, along the invariant manifold $C^{-1}(\{c\})$.

In order to apply the above stabilization criterion, let us recall first from Theorem \ref{t22} that the rattleback system \eqref{sys} is a Hamiltonian system of type \eqref{EPR} where:
\begin{align*}
 \nu (x,y,z) &:= \dfrac{1}{y^{\lambda -1}} ,~\forall (x,y,z)\in \Omega:=\{(x,y,z)\in \mathbb{R}^{3}: y\neq 0 \},\\
 H(x,y,z)&:= xy^{\lambda},~\forall (x,y,z)\in \Omega,\\
 C(x,y,z)&:=\dfrac{1}{2}\left(x^2 +y^2 +z^2 \right),~\forall (x,y,z)\in \Omega.
\end{align*}
Note that the maximal set where $H$ and $C$ are functionally independent is the open set
\begin{align*}
V:=\mathbb{R}^{3}\setminus\{\{(x,y,z)\in\mathbb{R}^{3}: y = 0\} & \cup\{(M,- M\sqrt{\lambda},0):M\in\mathbb{R}\}\\
&\cup\{(M,M\sqrt{\lambda},0):M\in\mathbb{R}\}\cup \{(0,0,M):M\in\mathbb{R}\}\}.
\end{align*}
Now the conclusion follows from the above mentioned stabilization criterion, by setting $\alpha(x,y,z):=\varepsilon, ~\forall (x,y,z)\in V$, and taking into account that for every $(h,c)\in \Sigma^{p,-}\cup\Sigma^{p,+}$, we have that $(h,c)$ is a regular value of $(H,C)=\mathcal{EC}$, $c\neq 0$ (hence $c$ is a regular value of $C$), and $\mathcal{EC}^{-1}(\{(h,c)\})\subset V$ (i.e. both periodic orbits $\mathcal{EC}^{-1}(\{(h,c)\})$ are contained in $V$).  

In order to complete the proof, note that from the relation \eqref{X0} we get
\begin{align*}
&-\alpha (x,y,z) (H(x,y,z)-h) \left[\nabla C(x,y,z) \times (\nabla H(x,y,z) \times \nabla C(x,y,z))\right]\\
&=\varepsilon  (xy^{\lambda}-h) \left[ \nabla C(x,y,z) \times (\nabla C(x,y,z) \times \nabla H(x,y,z))\right]\\
&=\varepsilon  (xy^{\lambda}-h) y^{\lambda-1} \left[ \begin{array}{c} y(\lambda x^2 -y^2 -z^2) \\ x(-\lambda x^2 +y^2 -\lambda z^2) \\ (\lambda +1)xyz \end{array} \right], ~\forall (x,y,z)\in V.
\end{align*}
\end{proof}

\begin{remark}
The stabilization procedure given in Theorem \ref{spor} still holds true if we replace the parameter $\varepsilon >0$ by any smooth function $\alpha \in C^{\infty}(V,(0,\infty))$.
\end{remark}

The phase portrait of the perturbed system $p$ in accordance with the geometric location of the pair of periodic orbits $\mathcal{EC}^{-1}(\{(h,c)\})$, $(h,c)\in \Sigma^{p,-}\cup\Sigma^{p,+}$, is illustrated in the Fig. \ref{tabp}. As in the case of equilibria, the geometric location of the periodic orbits depends on the parity of $\lambda$.

\begin{figure}[H]
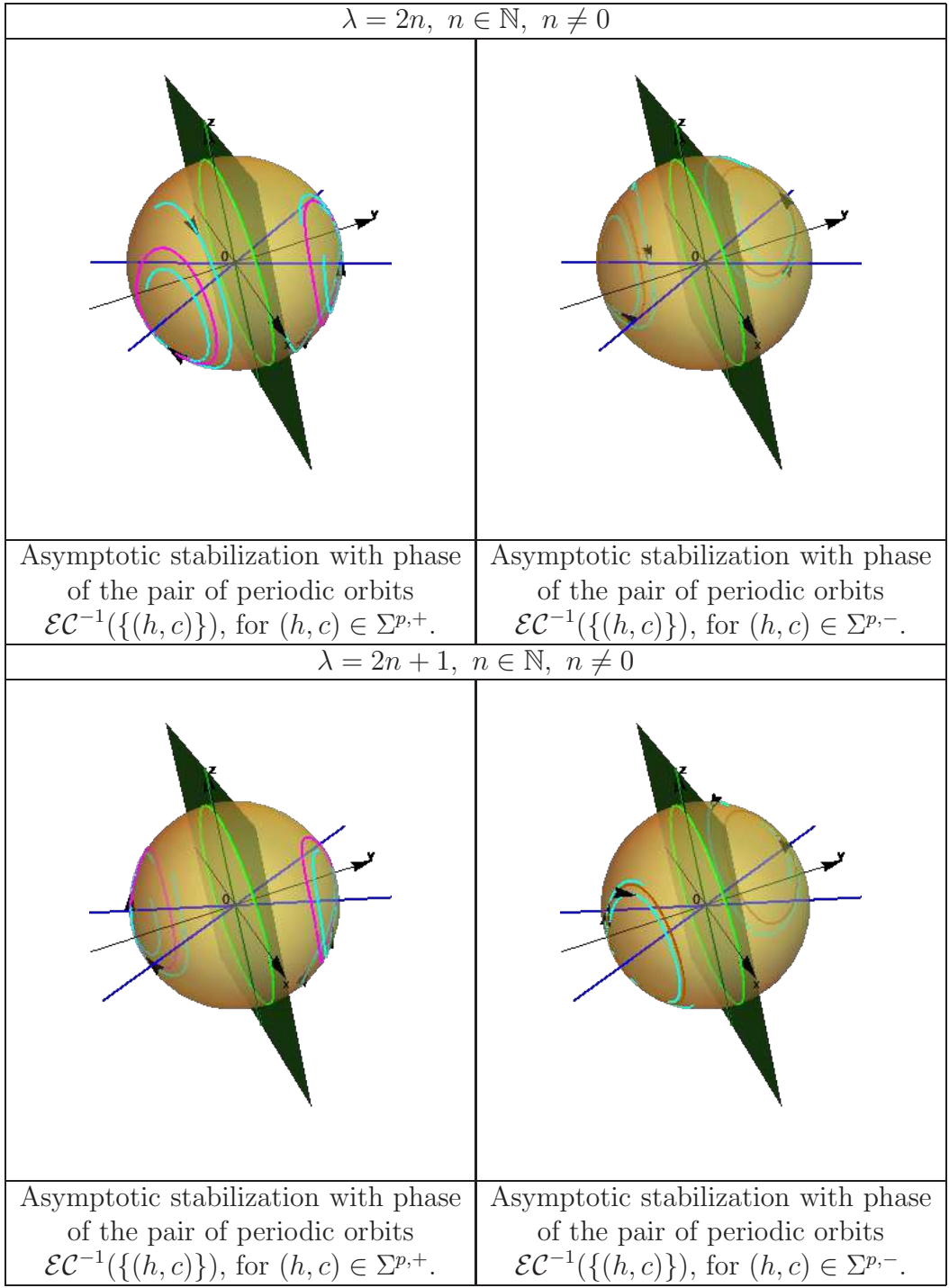

\centering
\vspace*{20pt}
\index{\footnote{}}
\begin{tabular}{|c|c|}
\hline
\multicolumn{2}{|c|}{$\lambda=2n,\ n\in\N,\ n\ne0$}\\
\hline
 & \\
\scalebox{0.7}{%
\includegraphics*{pert_orb_par1.%
eps}} &
\scalebox{0.7}{%
\includegraphics*{pert_orb_par2.%
eps}} \\
 &  \\
\hline
Asymptotic stabilization with phase   & Asymptotic stabilization with phase  \\
of the pair of periodic orbits  & of the pair of periodic orbits\\
 $\mathcal{EC}^{-1}(\{(h,c)\})$, for $(h,c)\in\Sigma^{p,+}$. & $\mathcal{EC}^{-1}(\{(h,c)\})$, for $(h,c)\in\Sigma^{p,-}$.\\
\hline
\multicolumn{2}{|c|}{$\lambda=2n+1,\ n\in\N,\ n\ne0$}\\
\hline
 & \\
\scalebox{0.7}{%
\includegraphics*{pert_orb_impar1.%
eps}} &
\scalebox{0.7}{%
\includegraphics*{pert_orb_impar2.%
eps}} \\
 &  \\
\hline
Asymptotic stabilization with phase   & Asymptotic stabilization with phase  \\
of the pair of periodic orbits  & of the pair of periodic orbits\\
 $\mathcal{EC}^{-1}(\{(h,c)\})$, for $(h,c)\in\Sigma^{p,+}$. & $\mathcal{EC}^{-1}(\{(h,c)\})$, for $(h,c)\in\Sigma^{p,-}$.\\
\hline
\end{tabular}
\caption{\it Phase portrait of the perturbed system $p$.}\label{tabp}
\end{figure}

\subsection{Asymptotic stabilization of the heteroclinic cycles of the rattleback system}

The aim of this section is to provide a method to stabilize asymptotically the heteroclinic cycles generated by pairs of antipodal unstable equilibrium states of the rattleback system \eqref{sys}, and also to keep the conservative nature of the dynamics, in the sense that the Casimir invariant $C = \dfrac{1}{2} (x^2 +y^2 +z^2)$ remains a first integral of the perturbed system. Thus, as already mentioned before, the asymptotic stabilization procedure makes sense only for perturbations along the corresponding level sets of $C$.

Before stating the main result of this section, let us recall from Proposition \ref{topostrt} that each point $(h,c)\in\Sigma^{u}$ corresponds to a pair of unstable equilibrium states connected by four heteroclinic orbits of the rattleback system \eqref{sys}, i.e. all together form the invariant set $\mathcal{EC}^{-1}(\{(h,c)\})$. Recall also from Theorem \ref{stability} that each unstable equilibrium state is given by $(0,0,M)$ for some $M\in\mathbb{R}\setminus\{0\}$.

Consequently, on each Casimir level set, $C^{-1}(\{M^2 /2\})$, $M\neq 0$, there exist exactly two unstable equilibrium states of the rattleback system \eqref{sys} (i.e. the antipodal points $-e_{3}^{|M|}$, and $e_{3}^{|M|}$, where $e_{3}^{|M|}:=(0,0,|M|)$) connected by four heteroclinic orbits. Note that, any two heteroclinic orbits together with the equilibrium states they connect (i.e. $-e_{3}^{|M|}$, $e_{3}^{|M|}$), form a heteroclinic cycle, located on the invariant sphere $C^{-1}(\{M^2 /2\})$.

In the following, for each pair of antipodal unstable equilibrium states of the rattleback system, $\{-e_{3}^{|M|}, e_{3}^{|M|}\}$, we construct a perturbation which stabilize asymptotically the heteroclinic cycles generated by $\{-e_{3}^{|M|}, e_{3}^{|M|}\}$, with respect to perturbations along the invariant manifold $C^{-1}(\{M^2 /2\})$. As the perturbations we consider are taken along the whole sphere $C^{-1}(\{M^2 /2\})$, the only candidate to asymptotic stabilization is the invariant set consisting of all heteroclinic cycles generated by $\{-e_{3}^{|M|}, e_{3}^{|M|}\}$, i.e. the set $\mathcal{EC}^{-1}(\{(0,M^2 /2)\})$. In contrast to the perturbations used in order to stabilize asymptotically the Lyapunov stable equilibrium states, and the periodic orbits of the rattleback system, this time the perturbation is not depending on $M$. Thus, the same perturbation stabilize asymptotically each invariant set $\mathcal{EC}^{-1}(\{(0,M^2 /2)\})$, for $M\in\mathbb{R}\setminus\{0\}$. Let us state now the main result of this section.

\begin{theorem}
Let $\varepsilon >0$ be a real parameter and let $M\in\mathbb{R}\setminus\{0\}$. The following perturbation of the rattleback system \eqref{sys} 
\begin{equation*}
p_{het}: \left[ \begin{array}{c} \dot{x} \\ \dot{y} \\ \dot{z} \end{array} \right]=\left[ \begin{array}{c}  \lambda xz \\ -yz \\ y^2 -\lambda x^2 \end{array} \right]+\varepsilon xy^{2\lambda -1}\left[ \begin{array}{c} y(\lambda x^2 -y^2 -z^2) \\ x(-\lambda x^2 +y^2 -\lambda z^2) \\ (\lambda +1)xyz \end{array} \right],~(x,y,z)\in\mathbb{R}^{3},
\end{equation*}
asymptotically stabilizes the set of heteroclinic cycles of the rattleback system \eqref{sys} generated by the unstable equilibrium states $-e_{3}^{|M|}, e_{3}^{|M|}$, with respect to perturbations along the invariant manifold $C^{-1}(\{M^2 /2\})$. 
\end{theorem}
\begin{proof}
First of all, note that (by construction, \eqref{X0}) each dynamical element of $\mathcal{EC}^{-1}(\{(0,M^2 /2)\})$, preserves its nature when seen as a dynamical element of the perturbed system $p_{het}$.

The rest of the proof follows from the classical LaSalle's stability criterion. In order to apply it, let $L:\mathbb{R}^3 \rightarrow \mathbb{R}$ be a smooth function given by $L(x,y,z):=x^2 y^{2\lambda},~ \forall (x,y,z)\in\mathbb{R}^3$. Denoting by $X_{het}$ the vector field which generates the perturbed system $p_{het}$, after some straightforward computations we obtain the following relation
\begin{align}\label{hetcon}
\mathcal{L}_{X_{het}}L(x,y,z)= -2\varepsilon x^2 y^{2(2\lambda -1)}\left[(y^2-\lambda x^2)^2 +y^2 z^2 + \lambda ^2 x^2 z^2 \right] \leq 0, ~\forall (x,y,z)\in\mathbb{R}^3.
\end{align} 
Using the above inequality and the fact that $C=\dfrac{1}{2}(x^2 +y^2 +z^2)$ is a first integral of $X_{het}$, we obtain that for every $\beta >0$, the set
$$
U_{\beta}:=\left\{(x,y,z)\in\mathbb{R}^3 : x^2 y^{2\lambda}\leq \beta\right\}\cap\left\{(x,y,z)\in\mathbb{R}^3 : x^2 +y^2 +z^2 \leq M^2 \right\},
$$
is compact and positively invariant. 

Denoting $\beta_0:= \dfrac{M^{2(\lambda +1)}\lambda^\lambda}{(\lambda +1)^{\lambda +1}}$, then for every $\beta \in (0,\beta_0)$, the set $U_{\beta}$ does not contain the equilibrium states $\pm e_{-}^{|M^{\prime}|}, \pm e_{+}^{|M^{\prime}|} \in C^{-1}(\{M^2 /2\})$, where $M^{\prime}:=\dfrac{M}{\sqrt{\lambda +1}}$.

Let us fix now some $\beta \in (0,\beta_0)$. Using LaSalle's stability criterion and taking into account that $C$ is a first integral of $X_{het}$, it follows that for every $(x_0,y_0,z_0)\in C^{-1}(\{M^2 /2\})\cap U_{\beta}$, the solution $(x(t),y(t),z(t))$ of the perturbed system $p_{het}$ starting from $(x_0,y_0,z_0)$ at $t=0$, approaches (as $t\rightarrow \infty$) the largest dynamically invariant subset of 
$$
\left\{ (x,y,z)\in U_{\beta} \cap C^{-1}(\{M^2 /2\}): \mathcal{L}_{X_{het}}L(x,y,z)=0\right\}.
$$
Using the relation \eqref{hetcon}, it follows that 
$$
\mathcal{L}_{X_{het}}L(x,y,z) =0 \Leftrightarrow (x,y,z)\in \{(x,y,z)\in\mathbb{R}^3 : x=0\}\cup \{(x,y,z)\in\mathbb{R}^3 : y=0\} \cup \mathcal{E},
$$
where $\mathcal{E}$ stands for the set of equilibrium states of the rattleback system \eqref{sys}.

As $\beta \in (0,\beta_0)$, the only equilibrium states that belong to $U_{\beta} \cap C^{-1}(\{M^2 /2\})$ are $-e_{3}^{|M|}$ and $ e_{3}^{|M|}$, and consequently, the largest dynamically invariant subset of 
$$
\left\{ (x,y,z)\in U_{\beta} \cap C^{-1}(\{M^2 /2\}): \mathcal{L}_{X_{het}}L(x,y,z)=0\right\}
$$
is precisely the set of heteroclinic cycles generated by the equilibria $-e_{3}^{|M|}$ and $e_{3}^{|M|}$, i.e. $\mathcal{EC}^{-1}(\{(0,M^2 /2)\})$.

Note that for every $\beta_1, \beta_2 \in (0,\beta_0)$ such that $\beta_1 < \beta_2$, we obtain
$$
C^{-1}(\{M^2 /2\})\cap U_{\beta_1} \subset C^{-1}(\{M^2 /2\})\cap U_{\beta_2}.
$$
Moreover, for $\beta =0$ we get that
\begin{align*}
C^{-1}(\{M^2 /2\})\cap U_{0} &=C^{-1}(\{M^2 /2\})\cap\left(\{(x,y,z)\in\mathbb{R}^3 : x=0\}\cup \{(x,y,z)\in\mathbb{R}^3 : y=0\}\right) \\
&=\mathcal{EC}^{-1}(\{(0,M^2 /2)\})\subset C^{-1}(\{M^2 /2\})\cap U_{\beta},~\forall \beta\in (0,\beta_0),
\end{align*}
whereas for $\beta=\beta_0$ we obtain that
$$
C^{-1}(\{M^2 /2\})\cap U_{\beta_0}=C^{-1}(\{M^2 /2\})\supset \{-e_{-}^{|M^{\prime}|},e_{-}^{|M^{\prime}|},-e_{+}^{|M^{\prime}|},e_{+}^{|M^{\prime}|}\}, ~ \text{where} ~ M^{\prime}:=\dfrac{M}{\sqrt{\lambda +1}}.
$$
Summarizing, we have shown that given any $\beta \in [0,\beta_0)$, then for every $(x_0,y_0,z_0)\in C^{-1}(\{M^2 /2\})\cap U_{\beta}$, the solution $(x(t),y(t),z(t))$ of the perturbed system $p_{het}$ starting from $(x_0,y_0,z_0)$ at $t=0$, approaches $\mathcal{EC}^{-1}(\{(0,M^2 /2)\})$ as $t\rightarrow \infty$.
\end{proof}

The phase portrait of the perturbed system $p_{het}$ is illustrated in the Fig. \ref{tabhet}.

\begin{figure}[H]
\vspace*{25pt}
\centering
\begin{tabular}{|c|}
\hline
\scalebox{0.7}{\includegraphics*{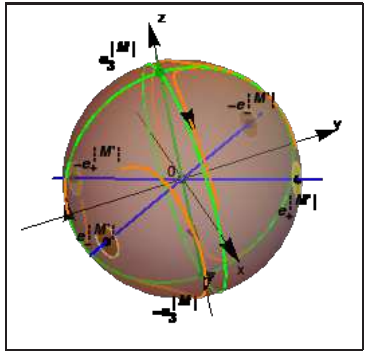}}\\
\hline
\end{tabular}
\caption{\it Phase portrait of the perturbed system $p_{het}$.}\label{tabhet}
\end{figure}

%\subsection*{Acknowledgment}
%This work was supported by a grant of the Romanian National Authority for Scientific Research, CNCS-UEFISCDI, project number %PN-II-RU-TE-2011-3-0103.

\bigskip
\bigskip

\noindent {\sc R.M. Tudoran}\\
West University of Timi\c soara,\\
Faculty of Mathematics and Computer Science,\\
Department of Mathematics,\\
Blvd. Vasile P\^arvan, No. 4,\\
300223-Timi\c soara, Rom\^ania.\\
E-mail: {\sf razvan.tudoran@e-uvt.ro}\\
\medskip

\noindent {\sc A. G\^\i rban}\\
"Politehnica" University of Timi\c soara,\\
Department of Mathematics, \\
Pia\c ta Victoriei, No. 2,\\
300006-Timi\c soara, Rom\^ania.\\
E-mail: {\sf anania.girban@gmail.com}
\medskip

\end{document}